\documentclass[11pt, a4paper]{article}

\usepackage{indentfirst}
\usepackage{lineno}
\usepackage{graphicx}
\usepackage{ae}
\usepackage{amsmath}
\usepackage{amssymb}
\usepackage{latexsym}
\usepackage{url}
\usepackage{epsfig}
\usepackage{cite}
\usepackage{mathrsfs}
\usepackage{amsfonts}
\usepackage{amsthm}
\usepackage{float}
\usepackage{longtable}
\allowdisplaybreaks[3]
\usepackage[numbers,sort&compress]{natbib}
\usepackage{caption}
\usepackage{CJK}

\allowdisplaybreaks  

\def\qed{\hfill$\Box$\vspace{12pt}}

\usepackage{color}

\long\def\delete#1{}

\newcommand{\be}{\begin{equation}}
\newcommand{\ee}{\end{equation}}
\newcommand{\ben}{\begin{equation*}}
\newcommand{\een}{\end{equation*}}
\newcommand{\bea}{\begin{eqnarray}}
\newcommand{\eea}{\end{eqnarray}}
\newcommand{\bean}{\begin{eqnarray*}}
\newcommand{\eean}{\end{eqnarray*}}

\def\det{{\rm det}}

\def\tr{{\rm tr}}

\newtheorem{thm}{Theorem}[section]
\newtheorem{cor}[thm]{Corollary}

\newtheorem{exam}[thm]{Example}
\newtheorem{lem}[thm]{Lemma}
\newtheorem{prop}[thm]{Proposition}
\newtheorem{prob}[thm]{Problem}

\newtheorem{rem}[thm]{Remark}
\numberwithin{equation}{section}

\title{\textbf{On the $A_{\alpha}$-characteristic polynomial of a graph}}

\author{Xiaogang Liu\thanks{Supported by the National Natural Science Foundation of China (Nos. 11361033 and 11601431), the China Postdoctoral Science Foundation (No. 2016M600813), the Natural Science Foundation of Shaanxi Province (No. 2017JQ1019)  and the Scientific Research Foundation of NPU (No. 3102016OQD029).}\\
{\small Department of Applied Mathematics}\\[-0.8ex]
{\small Northwestern Polytechnical University}\\[-0.8ex]
{\small Xi'an, Shaanxi 710072, P.R. China}\\
\emph{{\small \tt xiaogliu@nwpu.edu.cn}}
\and
Shunyi Liu\thanks{Supported by the National Natural Science Foundation of China (Nos. 11501050 and 11401044) and the Natural Science Basic Research Plan in Shaanxi Province of China (No. 2016JM6081).}\\
{\small School of Science}\\[-0.8ex]
{\small Chang'an University}\\[-0.8ex]
{\small Xi'an, Shaanxi 710064, P.R. China}\\
\emph{{\small \tt liu@chd.edu.cn}} }
\date{}

\begin{document}
\begin{CJK*}{GBK}{song}
\openup 0.5\jot
\maketitle

\begin{abstract}
Let $G$ be a graph with $n$ vertices, and let $A(G)$ and $D(G)$ denote respectively the adjacency matrix and the degree matrix of $G$. Define
$$
A_{\alpha}(G)=\alpha D(G)+(1-\alpha)A(G)
$$
for any real $\alpha\in [0,1]$. The \emph{$A_{\alpha}$-characteristic polynomial} of $G$ is defined to be
$$
\det(xI_n-A_{\alpha}(G))=\sum_jc_{\alpha j}(G)x^{n-j},
$$
where $\det(*)$ denotes the determinant of $*$, and $I_n$ is the identity matrix of size $n$. The \emph{$A_{\alpha}$-spectrum} of $G$ consists of all roots of the $A_{\alpha}$-characteristic polynomial of $G$. A graph $G$ is said to be \emph{determined by its $A_{\alpha}$-spectrum} if all graphs having the same $A_{\alpha}$-spectrum as $G$ are isomorphic to $G$.

In this paper, we first formulate the first four coefficients $c_{\alpha 0}(G)$, $c_{\alpha 1}(G)$, $c_{\alpha 2}(G)$ and  $c_{\alpha 3}(G)$ of the $A_{\alpha}$-characteristic polynomial of $G$. And then, we observe that $A_{\alpha}$-spectra are much efficient for us to distinguish graphs, by enumerating the $A_{\alpha}$-characteristic polynomials for all graphs on at most 10 vertices. To verify this observation, we characterize some graphs determined by their $A_{\alpha}$-spectra.
\bigskip

\noindent\textbf{Keywords:} Adjacency matrix; Degree matrix; $A_{\alpha}$-characteristic polynomial; $A_{\alpha}$-spectrum; Determined by its $A_{\alpha}$-spectrum

\bigskip

\noindent{{\bf AMS Subject Classification (2010):} 05C50}
\end{abstract}


\section{Introduction}
Let $G=(V(G),E(G))$ be a graph with the vertex set $V(G)=\{v_1,v_2,\ldots,v_n\}$ and the edge set $E(G)=\{e_1,e_2,\ldots,e_m\}$. The \emph{adjacency matrix} of $G$, denoted by $A(G)=(a_{ij})_{n\times n}$, is an $n\times n$
symmetric matrix such that $a_{ij}=1$ if vertices $v_i$ and $v_j$
are adjacent and $0$ otherwise. Let $d_{i}=d(v_i)=d_G(v_i)$ be the degree of
vertex $v_i$ in $G$. The \emph{degree matrix} of $G$, denoted by $D(G)$, is
the diagonal matrix with diagonal entries the vertex degrees of $G$. The \emph{Laplacian matrix} and the \emph{signless Laplacian matrix} of $G$ are defined as
$L(G)=D(G)-A(G)$ and $Q(G)=D(G)+A(G)$, respectively.

In \cite{Nikiforov16},  Nikiforov proposed to study the following matrix:
$$
A_{\alpha}(G)=\alpha D(G)+(1-\alpha)A(G),
$$
where $\alpha\in [0,1]$ is a real number. Note that $A_0(G)=A(G)$ and $2A_{1/2}(G)=Q(G)$. So, it was  claimed in \cite{Nikiforov16, NikiPasten17} that the matrices $A_{\alpha}(G)$ can underpin a unified theory of $A(G)$ and $Q(G)$. Up until now, a few properties on $A_{\alpha}(G)$ have been investigated, including bounds on the $k$-th largest  (especially, the largest, the second largest and the smallest)  eigenvalue of $A_{\alpha}(G)$ \cite{Nikiforov16, NikiPasten17}, the positive semidefiniteness of $A_{\alpha}(G)$ \cite{Nikiforov16, NikiRojo17}, etc. For more properties on $A_{\alpha}(G)$, we refer the reader to \cite{Nikiforov16}.

Let $M$ be an $n \times n$ real matrix. Denote by
$$
\phi(M;x)=\det(xI_n-M),
$$
or simply $\phi(M)$, the \emph{characteristic polynomial} of $M$, where $I_n$ is the identity matrix of size $n$. In particular, we call $\phi(A_{\alpha}(G);x)$ (respectively, $\phi(A(G);x)$, $\phi(L(G);x)$, or $\phi(Q(G);x)$) the \emph{$A_{\alpha}$-characteristic} (respectively, \emph{$A$-characteristic}, \emph{$L$-characteristic}, or \emph{$Q$-characteristic}) \emph{polynomial} of a graph $G$.

Let $G$ be a graph on $n$ vertices, and let $t_G$ denote the number of triangles in $G$. Suppose that
$\phi(A(G);x)=\sum_ja_{j}(G)x^{n-j}$. By Sachs' Coefficient Theorem (see \cite[Corollary 2.3.3]{Cvetkovic10}; or \cite{Sachs64}), one can easily verify the following result.

\begin{prop}\label{Acoeff1}
Let $G$ be a graph with $n$ vertices and $m$ edges. Then the first four coefficients in $\phi(A(G);x)$ are
\begin{align*}
&a_0(G)=1, \qquad  a_1(G)=0, \qquad  a_2(G)=-m,  \qquad a_{3}(G)=-2t_G.
 \end{align*}
\end{prop}

Suppose that $\phi(L(G);x)=\sum_jl_{j}(G)x^{n-j}$. In \cite{Oliveira02}, Oliveira et al. formulated the first four coefficients $l_{0}(G)$, $l_{1}(G)$, $l_{2}(G)$ and  $l_{3}(G)$ of $\phi(L(G);x)$, stated as follows.

\begin{prop}\label{Lcoeff1}\emph{(see \cite{Oliveira02})}
Let $G$ be a graph with $n$ vertices and $m$ edges and let
$\mathrm{deg}(G)=(d_1,d_2,\ldots,d_n)$ be its degree
sequence. Then the first four coefficients in $\phi(L(G);x)$ are
\begin{align*}
&l_0(G)=1, \qquad  l_1(G)=-2m, \qquad  l_2(G)=2m^2-m-\frac{1}{2}\sum_{i=1}^nd_i^2,   \\
 &l_{3}(G)= \frac{1}{3}\left(-4m^3+6m^2+3m\sum_{i=1}^{n}d_i^2-\sum_{i=1}^nd_i^3-3\sum_{i=1}^nd_i^2+6t_G\right).
 \end{align*}
\end{prop}

In this paper, we formulate the first four coefficients of the $A_{\alpha}$-characteristic polynomial of a graph $G$. We state our result as follows, where $t_G$ denotes the number of triangles in $G$.

\begin{thm}\label{Aalphacoeff1}
Let $G$ be a graph with $n$ vertices and $m$ edges, and let
$\mathrm{deg}(G)=(d_1,d_2,\dots ,d_n)$ be its degree
sequence. Suppose that $\phi(A_{\alpha}(G);x)=\sum_jc_{\alpha j}(G)x^{n-j}$. Then
 \begin{align*}
   & c_{\alpha 0}(G)=1,\qquad c_{\alpha 1}(G)=-2\alpha m,\qquad c_{\alpha 2}(G)=2\alpha^2m^2-(1-\alpha)^2m-\frac{1}{2}\alpha^2\sum_rd_r^2,\\
   & c_{\alpha 3}(G)=-\frac{1}{3}\left(6(1-\alpha)^3t_G-6\alpha(1-\alpha)^2m^2+3\alpha(1-\alpha)^2\sum_rd_r^2 +\alpha^3\left(4m^3-3m\sum_rd_r^2+\sum_rd_r^3\right)\right).
 \end{align*}
\end{thm}

Let $M$ be an $n\times n$ real symmetric matrix. Denote the eigenvalues of $M$ by $\lambda_1(M)\geq\lambda_2(M)\geq\cdots\geq\lambda_n(M)$.
The collection of eigenvalues of $M$ together with multiplicities are called the \emph{spectrum} of $M$. In particular, the spectrum of $A(G)$ (respectively, $L(G)$, $Q(G)$, $A_{\alpha}(G)$) is called the \emph{$A$-spectrum} (respectively, \emph{$L$-spectrum}, \emph{$Q$-spectrum}, \emph{$A_{\alpha}$-spectrum}) of $G$. Two graphs are said to be \emph{$A$-cospectral} if they have the same $A$-spectrum (equivalently, the same $A$-characteristic polynomial). A graph is called an \emph{$A$-DS graph} if it is \emph{determined by its $A$-spectrum}, meaning that there exists no other graph that is non-isomorphic to it but $A$-cospectral with it. Similar terminology will be used for $L(G)$, $Q(G)$ and $A_{\alpha}(G)$. So we can speak of \emph{$L$-cospectral graphs}, \emph{$Q$-cospectral graphs}, \emph{$A_{\alpha}$-cospectral graphs},  \emph{$L$-DS graphs},  \emph{$Q$-DS graphs} and \emph{$A_{\alpha}$-DS graphs}.

Characterizing which graphs are determined by their spectra is a classical but difficult problem in spectral graph theory which was raised by G\"{u}nthard and Primas \cite{GunthardP56} in 1956 with motivations from chemistry. Up until now, although many graphs have been proved to be DS graphs (see \cite{DamH03,DamH09}), the problem of determining DS  graphs is still far from being completely solved. Therefore, finding new families of DS graphs deserves further attention in order to enrich our database of DS graphs. Unfortunately, even for some simple-looking graphs, it is often challenging to determine whether they are DS or not.

 Let $J$ denote the matrix with all entries equal to one. In \cite[Concluding remarks]{DamH03}, van Dam and Haemers proposed to solve the following problem:

\begin{prob}\label{prob:1}
Which linear combination of $D(G)$, $A(G)$, and $J$ gives the most DS graphs?
\end{prob}

From \cite[Table 1]{DamH03}, van Dam and Haemers claimed that the signless Laplacian matrix $Q(G)=D(G)+A(G)$ would be a good candidate. Since then, a lot of researchers tried to confirm this claim. However, it seems to be negative, since it is so difficult to check whether a graph is DS or not.

In this paper, by enumerating the $A_{\alpha}$-characteristic polynomials for all graphs on at most 10 vertices, we observe that $A_{\alpha}$-spectra are much more efficient than $A$-spectra and $Q$-spectra when we use them to distinguish graphs. To verify this observation, we characterize some graphs determined by their $A_{\alpha}$-spectra.

The paper is organized as follows. In Section \ref{Sec:2}, we give some basic results which will be used to prove Theorem \ref{Aalphacoeff1}. In Section \ref{Sec:3}, we prove Theorem \ref{Aalphacoeff1}, and we also compute
the first four coefficients of $A_{\alpha}$-characteristic polynomials of some specific graphs as examples. In Section \ref{Sec:DS}, we enumerate the $A_{\alpha}$-characteristic polynomials for all graphs on at most 10 vertices, and we also characterize some graphs determined by their $A_{\alpha}$-spectra.

\section{Preliminaries}\label{Sec:2}
In this section, we mention some results, which will play an important role in the proof of Theorem \ref{Aalphacoeff1}.

Let $G$ be a graph with the vertex set $V(G)=\{v_1,v_2,\ldots,v_n\}$ and the edge set $E(G)=\{e_1,e_2,\ldots,e_m\}$. Define $\Gamma=(G,w)$ to be an edge-weighted graph obtained from $G$ by assigning each edge $e_i$ of $G$ with a non-zero weight $w(e_i)$. Usually, $w$ is called an edge-weight function of $\Gamma$. The adjacency matrix $A(\Gamma)$ of the edge-weighted graph $\Gamma=(G,w)$ is defined as the $n\times n$ matrix $A(\Gamma) = (a_{ij})$ with
$$
a_{ij}=\left\{\begin{array}{rl}
               w(e)\not=0,  &  \text{if $v_iv_j$ is an edge $e$ of $G$,} \\[0.2cm]
               0,                    &  \text{otherwise.}
             \end{array}
\right.
$$

Suppose that $v_{r_1},v_{r_2},\ldots,v_{r_k}$ are $k$ distinct vertices of $\Gamma$. Denote by $\Gamma_{r_1,r_2,\ldots,r_k}$ the edge-weighted subgraph obtained from $\Gamma$ by deleting these $k$ vertices, together with all weighted edges incident to these vertices, from $\Gamma$. Denote by $\Gamma[h_{r_1},h_{r_2},\ldots,h_{r_k}]$ the weighted graph obtained from $\Gamma$ by adding a loop of weight $h_{r_j}$ to each vertex $v_{r_j}$ for $j=1, 2,\ldots, k$.

It is well known that if one column of a matrix $M$ is written as a sum $u+v$ of two column vectors, and all other columns are left unchanged, then the determinant of $M$ is the sum of the determinants of the matrices obtained from $M$ by replacing the column by $u$ and then by $v$. By applying this fact to $\phi(A(\Gamma[h_{r_1}]);x)$, we have the following recursion relation
$$
\phi(A(\Gamma[h_{r_1}]);x)=\phi(A(\Gamma);x)-h_{r_1}\phi(A(\Gamma_{r_1});x).
$$
By the above equation, we have
 \begin{align*}
 \phi(A(\Gamma[h_{r_1}, h_{r_2}]);x) & =\phi(A(\Gamma[h_{r_1}]);x) - h_{r_2}\phi(A(\Gamma_{r_2}[h_{r_1}]);x)  \\
   & =\Big(\phi(A(\Gamma);x)-h_{r_1}\phi(A(\Gamma_{r_1});x)\Big) -h_{r_2}\Big(\phi(A(\Gamma_{r_2});x)-h_{r_1}\phi(A(\Gamma_{r_1,r_2});x)\Big)\\
   &=\phi(A(\Gamma);x)-h_{r_1}\phi(A(\Gamma_{r_1});x) -h_{r_2}\phi(A(\Gamma_{r_2});x)+h_{r_1}h_{r_2}\phi(A(\Gamma_{r_1,r_2});x).
   \end{align*}
More generally, one can verify the following result.

\begin{lem}\label{PerWeight1}
Let $\Gamma=(G,w)$ be an edge-weighted graph with $n$ vertices, and $\Gamma[h_{1},h_{2},\ldots,h_{n}]$ the graph obtained from $\Gamma$ by adding a loop of weight $h_{j}$ to each vertex $v_{j}$ for $j=1, 2,\ldots, n$. Then
$$
\phi(A(\Gamma[h_1,h_2,\ldots,h_n]);x)=\phi(A(\Gamma);x)+\sum_{k=1}^n(-1)^k\sum_{1\le r_1<r_2<\cdots<r_k\le n}h_{r_1}h_{r_2}\cdots h_{r_k}\phi(A(\Gamma_{r_1,r_2,\ldots,r_k});x).
$$
\end{lem}

\begin{lem}\label{Mcoeff11}
Let $M$ be an $n\times n$ real matrix. Suppose that
$\phi(M;x)=\sum_jm_{j}(M)x^{n-j}$ is a monic polynomial. Then $\phi(kM;x)=\sum_jk^jm_{j}(M)x^{n-j}$, where $k$ is a real number.
\end{lem}
\begin{proof}
Suppose that $\lambda_1, \lambda_2, \ldots, \lambda_n$ are eigenvalues of $M$.  Then
\begin{align*}
\phi(M;x) &=\det(xI_n-M)\\
                       &=\prod_{i=1}^n(x-\lambda_i)\\
                       &=x^n-\sum_i \lambda_i x^{n-1}+\sum_{i<j}\lambda_i\lambda_jx^{n-2}-\sum_{i<j<k}\lambda_i\lambda_j\lambda_kx^{n-3}+\cdots,
 \end{align*}
and
\begin{align*}
\phi(kM;x) &=\det(xI_n-kM)\\
                       &=\prod_{i=1}^n(x-k\lambda_i)\\
                       &=x^n-k\sum_i \lambda_i x^{n-1}+k^2\sum_{i<j}\lambda_i\lambda_jx^{n-2}-k^3\sum_{i<j<k}\lambda_i\lambda_j\lambda_kx^{n-3}+\cdots.
 \end{align*}
By comparing the above equations, we obtain the required result.
\qed\end{proof}

Let $G$ be a graph with $n$ vertices and $m$ edges.  Recall that $t_G$ denotes the number of triangles in $G$. Define $\Gamma'$ to be an edge-weighted graph obtained from $G$ by assigning each edge $e_i$ of $G$ with a non-zero weight $1-\alpha$. Clearly, $A(\Gamma')=(1-\alpha)A(G)$. By Lemma \ref{Mcoeff11} and Proposition \ref{Acoeff1}, we have the following result immediately.

\begin{lem}\label{AWcoeff1}
Let $G$ and $\Gamma'$ be as above. Suppose that
$\phi(A(\Gamma');x)=\sum_ja_{j}(\Gamma')x^{n-j}$. Then the first four coefficients in $\phi(A(\Gamma');x)$ are
\begin{align*}
&a_0(\Gamma')=1, \qquad  a_1(\Gamma')=0, \qquad  a_2(\Gamma')=-(1-\alpha)^2m,  \qquad a_{3}(\Gamma')=-2(1-\alpha)^3t_G.
 \end{align*}

\end{lem}

\section{Coefficients of the $A_{\alpha}$-characteristic polynomial of a graph}\label{Sec:3}
\subsection{Proof of Theorem \ref{Aalphacoeff1}}
Let $G$ be a graph with $n$ vertices and $m$ edges, and let
$\mathrm{deg}(G)=(d_1,d_2,\dots ,d_n)$ be its degree
sequence. Recall that $\Gamma'$ denotes an edge-weighted graph obtained from $G$ by assigning each edge $e_i$ of $G$ with a non-zero weight $1-\alpha$. Then by Lemma \ref{PerWeight1}, we have
\begin{align}
  \phi(A_{\alpha}(G);x) &=\det(xI_n-A_{\alpha}(G)) \nonumber\\
  &=\det(xI_n-A(\Gamma'[\alpha d_1,\alpha d_2,\ldots,\alpha d_n])) \nonumber\\
  &=\phi(A(\Gamma');x)+\sum_{k=1}^n(-1)^k\alpha^k\sum_{1\le r_1<r_2<\cdots<r_k\le n}d_{r_1}d_{r_2}\cdots d_{r_k}\phi(A(\Gamma'_{r_1,r_2,\ldots,r_k});x). \label{Q-per-Pol1}
\end{align}

Recall that
$\phi(A(\Gamma');x)=\sum_ja_{j}(\Gamma')x^{n-j}$ and $\phi(A_{\alpha}(G);x)=\sum_jc_{\alpha j}(G)x^{n-j}$. By Equation (\ref{Q-per-Pol1}), we have
\begin{equation}\label{Q-Coeffi-1}
c_{\alpha j}(G)= a_j(\Gamma')+\sum_{k=1}^n(-1)^k\alpha^k\sum_{1\le r_1<r_2<\cdots<r_k\le n}d_{r_1}d_{r_2}\cdots d_{r_k}a_{j-k}(\Gamma'_{r_1,r_2,\ldots,r_k}).
\end{equation}

\begin{Tproof}\textbf{of Theorem \ref{Aalphacoeff1}.}~~
By Lemma \ref{AWcoeff1}, we have $a_0(\Gamma')=1, a_1(\Gamma')=0, a_2(\Gamma')=-(1-\alpha)^2m$ and $a_{3}(\Gamma')=-2(1-\alpha)^3t_G$. Note that $\sum_rd_r=2m$. By Equation (\ref{Q-Coeffi-1}), we have
\begin{align*}
  c_{\alpha 0}(G) &=a_0(\Gamma')=1;\\
  c_{\alpha 1}(G) & =a_1(\Gamma')-\alpha\sum_rd_ra_0(\Gamma'_r)=0-\alpha\sum_rd_r=-2\alpha m;\\
  c_{\alpha 2}(G)& =a_2(\Gamma')-\alpha\sum_rd_ra_1(\Gamma'_r) +\alpha^2\sum_{r<s}d_rd_sa_0(\Gamma'_{r,s})\\
              &=-(1-\alpha)^2m-0+\alpha^2\sum_{r<s}d_rd_s\\
              &=-(1-\alpha)^2m+\frac{1}{2}\alpha^2\left(\sum_r\sum_sd_rd_s-\sum_rd_r^2\right)\\
              &=-(1-\alpha)^2m+\frac{1}{2}\alpha^2(2m)^2-\frac{1}{2}\alpha^2\sum_rd_r^2\\
              &=2\alpha^2m^2-(1-\alpha)^2m-\frac{1}{2}\alpha^2\sum_rd_r^2;\\
  c_{\alpha 3}(G)& =a_3(\Gamma')-\alpha\sum_rd_ra_2(\Gamma'_r)+\alpha^2\sum_{r<s}d_rd_sa_1(\Gamma'_{r,s}) -\alpha^3\sum_{r<s<t}d_rd_sd_ta_0(\Gamma'_{r,s,t})\\          &=-2(1-\alpha)^3t_G+\alpha(1-\alpha)^2\sum_r d_r(m-d_r)-\alpha^3\sum_{r<s<t}d_rd_sd_t\\
  &=-2(1-\alpha)^3t_G+\alpha(1-\alpha)^2m\sum_rd_r-\alpha(1-\alpha)^2\sum_rd_r^2\\ &\quad-\frac{1}{6}\alpha^3\left(\sum_r\sum_s\sum_td_rd_sd_t-3\sum_r\sum_{\substack{s \\ s\not=r}}d_rd_s^2-\sum_rd_r^3 \right)\\
  &=-2(1-\alpha)^3t_G+\alpha(1-\alpha)^2m\sum_rd_r-\alpha(1-\alpha)^2\sum_rd_r^2\\ &\quad-\frac{1}{6}\alpha^3\left(\sum_r\sum_s\sum_td_rd_sd_t-3\sum_r\sum_sd_rd_s^2+2\sum_rd_r^3 \right)\\
  &=-2(1-\alpha)^3t_G+2\alpha(1-\alpha)^2m^2-\alpha(1-\alpha)^2\sum_rd_r^2 -\frac{1}{3}\alpha^3\left(4m^3-3m\sum_sd_s^2+\sum_rd_r^3\right)\\
  &=-\frac{1}{3}\left(6(1-\alpha)^3t_G-6\alpha(1-\alpha)^2m^2+3\alpha(1-\alpha)^2\sum_rd_r^2 +\alpha^3\left(4m^3-3m\sum_rd_r^2+\sum_rd_r^3\right)\right).
\end{align*}
This completes the proof. \qed\end{Tproof}

As a corollary of Theorem \ref{Aalphacoeff1}, by setting $\alpha=\frac{1}{2}$ and then applying Lemma \ref{Mcoeff11}, we immediately obtain the following result.

\begin{cor}\label{Qcoeff1}
Let $G$ be a graph with $n$ vertices and $m$ edges, and let
$\mathrm{deg}(G)=(d_1,d_2,\dots ,d_n)$ be its degree
sequence. Suppose that $\phi(Q(G);x)=\sum_jq_{j}(G)x^{n-j}$. Then
 \begin{align*}
   & q_{0}(G)=1,\qquad q_{1}(G)=-2m,\qquad q_{2}(G)=2m^2-m-\frac{1}{2}\sum_{r}d_r^2,\\
   & q_{3}(G)= -\frac{1}{3}\left(6t_G-6m^2+4m^3+3(1-m)\sum_rd_r^2+\sum_rd_r^3\right).
 \end{align*}
\end{cor}

\begin{proof}
Let $\alpha=\frac{1}{2}$. Then
\begin{align*}
   & c_{\frac{1}{2} 0}(G)=1,\qquad c_{\frac{1}{2} 1}(G)=-m,\qquad c_{\frac{1}{2} 2}(G)=\frac{1}{2}m^2-\frac{1}{4}m-\frac{1}{8}\sum_rd_r^2,\\
   & c_{\frac{1}{2} 3}(G)=-\frac{1}{3}\left(\frac{3}{4}t_G-\frac{3}{4}m^2+\frac{3}{8}\sum_rd_r^2 +\frac{1}{8}\left(4m^3-3m\sum_rd_r^2+\sum_rd_r^3\right)\right).
 \end{align*}
Note that $2A_{1/2}(G)=Q(G)$. By Lemma \ref{Mcoeff11}, we have
\begin{align*}
   q_{0}(G)&=c_{\frac{1}{2} 0}(G)=1,\qquad q_{1}(G)=2c_{\frac{1}{2} 1}(G)=-2m,\qquad q_{2}(G)=4c_{\frac{1}{2} 2}(G)=2m^2-m-\frac{1}{2}\sum_rd_r^2,\\
   q_{3}(G)&=8c_{\frac{1}{2} 3}(G)=-\frac{1}{3}\left(6t_G-6m^2+3\sum_rd_r^2 +4m^3-3m\sum_rd_r^2+\sum_rd_r^3\right)\\
   &=-\frac{1}{3}\left(6t_G-6m^2+4m^3+3(1-m)\sum_rd_r^2+\sum_rd_r^3\right).
 \end{align*}
This completes the proof.
\qed\end{proof}

In \cite{Nikiforov16}, Nikiforov proved the following results.

\begin{prop}\emph{(see \cite[Propositions 34 and 35]{Nikiforov16})}\label{prop:1}
Let $G$ be a graph with $n$ vertices and $m$ edges, and let
$\mathrm{deg}(G)=(d_1,d_2,\ldots,d_n)$ be the degree
sequence of $G$. Then
 \begin{align*}
 \sum_{i=1}^n\lambda_i &= \tr\left(A_{\alpha}(G)\right)=\alpha\sum_{i=1}^nd_i=2\alpha m,  \\
 \sum_{i=1}^n\lambda_i^2&=  \tr\left(A_{\alpha}^2(G)\right)= 2(1-\alpha)^2m+\alpha^2\sum_{i=1}^nd_i^2,
 \end{align*}
where $\lambda_i$'s are the eigenvalues of $A_{\alpha}(G)$ and $\tr(A_{\alpha}(G))$ means the trace of $A_{\alpha}(G)$.
\end{prop}

Similarly, we obtain a formula for the sum of the cubes of the $A_{\alpha}$-eigenvalues.

\begin{prop}\label{prop:2}
Let $G$, $\lambda_i$ and $d_i$ be as in Proposition \ref{prop:1}. Then
$$
\sum_{i=1}^n\lambda_i^3=\tr\left(A_{\alpha}^3(G)\right) =\alpha^3\sum_{i=1}^nd_i^3+3\alpha(1-\alpha)^2\sum_{i=1}^nd_i^2+6(1-\alpha)^3t_G.
$$
\end{prop}

\begin{proof}
Let $A:= A(G)$ and $D:= D(G)$. Then
\begin{align*}
 A_{\alpha}^3(G) =& \alpha^3D^3+\alpha^2(1-\alpha)DAD+\alpha^2(1-\alpha)AD^2+\alpha(1-\alpha)^2A^2D\\
   &+\alpha^2(1-\alpha)D^2A+\alpha(1-\alpha)^2 DA^2+\alpha(1-\alpha)^2ADA+(1-\alpha)^3A^3.
 \end{align*}
Taking the
trace of $A_{\alpha}^3(G)$, we have
\begin{align*}
   \tr\left(A_{\alpha}^3(G)\right) & =\alpha^3\tr\left(D^3\right)+3\alpha(1-\alpha)^2\tr\left(DA^2\right)+(1-\alpha)^3\tr\left(A^3\right)\\
   &=\alpha^3\sum_{i=1}^nd_i^3+3\alpha(1-\alpha)^2\sum_{i=1}^nd_i^2+6(1-\alpha)^3t_G.
 \end{align*}
This completes the proof.
\qed\end{proof}

Plugging Propositions \ref{prop:1} and \ref{prop:2} into Theorem \ref{Aalphacoeff1}, we have the following result.

\begin{cor}
Let $G$ and $\phi(A_{\alpha}(G);x)$ be as in Theorem \ref{Aalphacoeff1}. Then
\begin{align*}
   & c_{\alpha 0}(G)=1,\qquad c_{\alpha 1}(G)=-\sum_{i=1}^n\lambda_i =-\tr\left(A_{\alpha}(G)\right),\\
   &c_{\alpha 2}(G)=2\alpha^2m^2-\frac{1}{2}\sum_{i=1}^n\lambda_i^2 =2\alpha^2m^2-\frac{1}{2}\tr\left(A_{\alpha}^2(G)\right),\\
   & c_{\alpha 3}(G)=-\frac{1}{3}\left(\sum_{i=1}^n\lambda_i^3-3\alpha m\sum_{i=1}^n\lambda_i^2+4\alpha^3m^3\right) =-\frac{1}{3}\Big(\tr\left(A_{\alpha}^3(G)\right)-3\alpha m\, \tr\left(A_{\alpha}^2(G)\right)+4\alpha^3m^3\Big).
 \end{align*}
\end{cor}

\subsection{Examples}
In this subsection, we give the first four coefficients of $A_{\alpha}$-characteristic polynomials of some specific graphs.

\begin{exam}\label{Exam1}
{\em Let $P_n$ be the path with $n$ vertices. Clearly, it has $n-1$ edges, and the number of triangles in $P_n$ is $t_{P_n}=0$. By Theorem \ref{Aalphacoeff1}, we have
 \begin{align*}
& c_{\alpha 0}(P_n)=1,\quad c_{\alpha 1}(P_n)=-2(n-1)\alpha,\\
&c_{\alpha 2}(P_n)=(2n-3)(n-2) {\alpha}^{2}+ 2(n-1)\alpha-(n-1),\\
& c_{\alpha 3}(P_n)=  -\frac{2}{3}( 2n-5)(n-2)(n-3){\alpha}^{3} -4({n}-2)^2{\alpha}^{2}+2({n}-2)^2\alpha.
 \end{align*}
By Corollary \ref{Qcoeff1}, we have
\begin{align*}
 &q_{0}(P_n)=1,\quad q_{1}(P_n)=-2(n-1),\quad q_{2}(P_n)=(2n-3)(n-2), \\
 &q_{3}(P_n)= -\frac{2}{3}(2n-5)(n-2)(n-3).
 \end{align*}
}
\end{exam}

\begin{exam}\label{Exam2}
{\em Let $K_n$ be the complete graph with $n$ vertices. Clearly, it has $n(n-1)/2$ edges, and the number of triangles in $K_n$ is $t_{K_n}=n(n-1)(n-2)/6$. By Theorem \ref{Aalphacoeff1}, we have
 \begin{align*}
c_{\alpha 0}(K_n)=&\, 1,\quad c_{\alpha 1}(K_n)=-n(n-1)\alpha,\\
c_{\alpha 2}(K_n)=&\, \frac{1}{2}{n}^{2}(n-1)(n-2)\alpha^2+n(n-1)\alpha-\frac{1}{2}n(n-1),\\
c_{\alpha 3}(K_n)=& -\frac{1}{6}n^3(n-1)(n-2)(n-3)\alpha^3-n^2(n-1)(n-2)\alpha^2 \\
&+\frac{1}{2}n(n-1)(n-2)(n+1)\alpha-\frac{1}{3}n(n-1)(n-2).
 \end{align*}
By Corollary \ref{Qcoeff1}, we have
\begin{align*}
 &q_{0}(K_n)=1,\quad q_{1}(K_n)=-n(n-1),\quad q_{2}(K_n)=\frac{1}{2}n^2(n-1)(n-2), \\
 &q_{3}(K_n)=-\frac{1}{6}n(n+1)(n-1)(n-2)^3.
 \end{align*}
}
\end{exam}

\begin{exam}\label{Exam3}
{\em Let $C_n$ be the cycle with $n\ge4$ vertices. Clearly, it has $n$ edges, and the number of triangles in $C_n$ is $t_{C_n}=0$. By Theorem \ref{Aalphacoeff1}, we have
 \begin{align*}
&c_{\alpha 0}(C_n)=1,\quad c_{\alpha 1}(C_n)=-2n\alpha,\quad c_{\alpha 2}(G)=n(2n-3)\alpha^{2}+2n\alpha-n,\\
&c_{\alpha 3}(C_n)=-\frac{2}{3}n(n-2)(2n-5)\alpha^{3}-4n(n-2)\alpha^{2}+2n(n-2)\alpha.
 \end{align*}
By Corollary \ref{Qcoeff1}, we have
\begin{align*}
 q_{0}(C_n)=1,\quad q_{1}(C_n)=-2n,\quad q_{2}(C_n)=n(2n-3), \quad q_{3}(C_n)=-\frac{2}{3}n(n-2)(2n-5).
 \end{align*}
}
\end{exam}

\begin{exam}\label{Exam4}
{\em Let $F_n$ be the friendship graph with $2n+1$ vertices and $3n$ edges. The number of triangles in $F_n$ is $t_{F_n}=n$. By Theorem \ref{Aalphacoeff1}, we have
 \begin{align*}
&c_{\alpha 0}(F_n)=1,\quad c_{\alpha 1}(F_n)=-6n\alpha, \quad c_{\alpha 2}(F_n)=n(16n-7)\alpha^{2}+6n\alpha-3n,\\
&c_{\alpha 3}(F_n)=-\frac{2}{3}n(40n-17)(n-1)\alpha^{3}-2n(14n-5)\alpha^{2}+2n(7n-1)\alpha-2n.
 \end{align*}
By Corollary \ref{Qcoeff1}, we have
\begin{align*}
 q_{0}(F_n)=1,\quad q_{1}(F_n)=-6n,\quad q_{2}(F_n)=n(16n-7), \quad q_{3}(F_n)=-\frac{2}{3}n(40n^2-57n+23).
 \end{align*}
}
\end{exam}

\begin{exam}\label{Exam5}
{\em Let $W_n$ be the wheel graph with $n+1$ vertices and $2n$ edges. The number of triangles in $W_n$ is $t_{W_n}=n$. By Theorem \ref{Aalphacoeff1}, we have
 \begin{align*}
&c_{\alpha 0}(W_n)=1,\quad c_{\alpha 1}(W_n)=-4n\alpha,\quad c_{\alpha 2}(W_n)=\frac{1}{2}n(15n-13)\alpha^{2}+4n\alpha-2n,\\
&c_{\alpha 3}(W_n)=-n(n-1)(9n-16)\alpha^{3}-2n(7n-6)\alpha^{2}+n(7n-3)\alpha-2n.
 \end{align*}
By Corollary \ref{Qcoeff1}, we have
\begin{align*}
 q_{0}(W_n)=1,\quad q_{1}(W_n)=-4n,\quad q_{2}(W_n)=\frac{1}{2}n(15n-13), \quad q_{3}(W_n)=-n(9n^2-25n+20).
 \end{align*}
}
\end{exam}

\begin{exam}\label{Exam6}
{\em Let $K_{a,b}$ be the complete bipartite graph with partition sets of sizes $a$ and $b$. The number of triangles in $K_{a,b}$ is $t_{K_{a,b}}=0$. By Theorem \ref{Aalphacoeff1}, we have
 \begin{align*}
c_{\alpha 0}(K_{a,b})=&1,\quad c_{\alpha 1}(K_{a,b})=-2ab\alpha,\quad c_{\alpha 2}(K_{a,b})=\frac{1}{2}ab(4ab-a-b-2)\alpha^{2}+2ab\alpha-ab,\\
c_{\alpha 3}(K_{a,b}) =&-\frac{1}{3}ab(4{a}^{2}{b}^{2}-3{a}^{2}b-3a{b}^{2}+{a}^{2}-6ab+{b}^{2}+3a+3b)\alpha^{3}\\ &-2ab(2ab-a-b)\alpha^{2}+ab(2ab-a-b)\alpha.
 \end{align*}
By Corollary \ref{Qcoeff1}, we have
\begin{align*}
& q_{0}(K_{a,b})=1,\quad q_{1}(K_{a,b})=-2ab,\quad q_{2}(K_{a,b})=\frac{1}{2}ab(4ab-a-b-2), \\
& q_{3}(K_{a,b})=-\frac{1}{3}ab(4{a}^{2}{b}^{2}-3{a}^{2}b-3a{b}^{2}+{a}^{2}-6ab+{b}^{2}+3a+3b).
 \end{align*}
}
\end{exam}

\section{Graphs determined by the $A_{\alpha}$-spectra}\label{Sec:DS}

Recall that two graphs are said to be \emph{$A_{\alpha}$-cospectral} if they have the same $A_{\alpha}$-spectrum (i.e., the same $A_{\alpha}$-characteristic polynomial). A graph $H$ is called an \emph{$A_{\alpha}$-cospectral mate} of a graph $G$ if $H$ and $G$ are $A_{\alpha}$-cospectral but $H$ is not isomorphic to $G$. Recall also that a graph $G$ is called an \emph{$A_{\alpha}$-DS graph} if it is \emph{determined by its $A_{\alpha}$-spectrum}, meaning that $G$ has no $A_{\alpha}$-cospectral mate.

\subsection{Enumeration}\label{Sec:Enum}
In this subsection, we have enumerated the $A_{\alpha}$-characteristic polynomials for all graphs on at most 10 vertices, and count the number of graphs for which there exists at least one $A_{\alpha}$-cospectral mate.

To determine the $A_{\alpha}$-characteristic polynomials of graphs we first of all have to generate the graphs by computer. All graphs on at most 10 vertices are generated by nauty and Traces~\cite{McPi}. Then the $A_{\alpha}$-characteristic polynomials of these graphs are computed by a Maple procedure. Finally we count the number of $A_{\alpha}$-cospectral graphs.

The results are in Table~\ref{tab:Table 1}. This table lists for $n\le 10$ the total number of graphs on $n$ vertices, the total number of distinct $A_{\alpha}$-characteristic polynomials of such graphs, the number of such graphs with an $A_{\alpha}$-cospectral mate, the fraction of such graphs with an $A_{\alpha}$-cospectral mate, and the size of the largest family of $A_{\alpha}$-cospectral graphs.

\begin{table}[htb]
\centering
\captionsetup{singlelinecheck=off,skip=0pt}
\caption{Computational data on $n\le 10$ vertices}\vspace{0.2cm}
\label{tab:Table 1}
    \begin{tabular}{rrrrrr}\hline
    $n$ & $\#$graphs & $\#$ $A_{\alpha}$-char. pols & $\#$ with mate & frac. with mate & max. family\\ \hline
     1  & 1          & 1              & 0                      &0                & 1 \\
     2  & 2          & 2              & 0                      &0                & 1 \\
     3  & 4          & 4              & 0                      &0                & 1 \\
     4  & 11         & 11             & 0                      &0                & 1 \\
     5  & 34         & 34             & 0                      &0                & 1 \\
     6  & 156        & 156            & 0                      &0                & 1 \\
     7  & 1044       & 1044           & 0                      &0                & 1 \\
     8  & 12346      & 12346          & 0                      &0                & 1 \\
     9  & 274668     & 274667         & 2                      & 0.000007281     & 2 \\
     10 & 12005168   & 12000093       & 10146                  & 0.000845136     & 3 \\ \hline
    \end{tabular}
\end{table}

In Table~\ref{tab:Table 1}, we see that the smallest $A_{\alpha}$-cospectral graphs, with respect to the order, contain 9 vertices. There is exactly one pair of $A_{\alpha}$-cospectral graphs on 9 vertices (see Fig. \ref{fig:Figure_1}). The common $A_{\alpha}$-characteristic polynomials of $G$ and $H$ in Fig. \ref{fig:Figure_1} are
\begin{align*}
&\quad \phi(A_{\alpha}(G);x)=\phi(A_{\alpha}(H);x)\\
&=x^9-36\alpha x^8+(556\alpha^2+36\alpha-18)x^7-(4806\alpha^3+1042\alpha^2-542\alpha+14)x^6\\
&\quad+(25393\alpha^4+12578\alpha^3-6513\alpha^2+60\alpha+67)x^5\\
&\quad-(83826\alpha^5+81818\alpha^4-39908\alpha^3-2746\alpha^2+1634\alpha-60)x^4\\ &\quad+(168450\alpha^6+308434\alpha^5-130253\alpha^4-34080\alpha^3+14391\alpha^2-576\alpha-62)x^3\\ &\quad-(187812\alpha^7+669816\alpha^6-207004\alpha^5-159042\alpha^4+57080\alpha^3-188\alpha^2-1064\alpha+46)x^2\\ &\quad+(88560\alpha^8+768822\alpha^7-96837\alpha^6-331528\alpha^5+99686\alpha^4+9946\alpha^3-5173\alpha^2+272\alpha+12)x\\ &\quad-(354240\alpha^8+60948\alpha^7-256002\alpha^6+57690\alpha^5+20308\alpha^4-6952\alpha^3+54\alpha^2+122\alpha-8).
\end{align*}

\begin{figure}
\centering
\vspace{-7.5cm}
\includegraphics[scale=0.75]{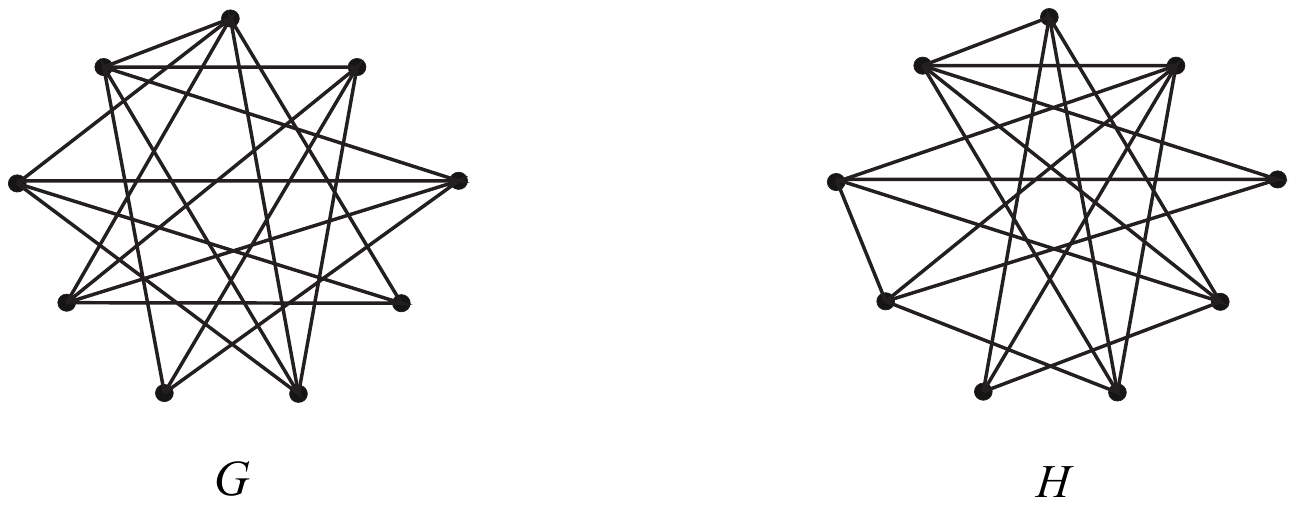}
\vspace{-9.5cm}
   \caption{The smallest pair of $A_{\alpha}$-cospectral graphs.}
   \protect\label{fig:Figure_1}
\end{figure}

By comparing the computational results of Table \ref{tab:Table 1} with that of \cite[Tables 1 and 2]{HaemersS04}, we see that $A_{\alpha}$-spectra are much more efficient than $A$-spectra and $Q$-spectra when we use them to distinguish graphs.

\subsection{$A_{\alpha}$-DS graphs}\label{Sec:Appl}

In this subsection, we give some graphs determined by the $A_{\alpha}$-spectra. Recall that $A_0(G)=A(G)$ and $2A_{1/2}(G)=Q(G)$. Thus, properties determined by the $A$-spectrum (respectively, $Q$-spectrum) can also be determined by the $A_{\alpha}$-spectrum.

\begin{thm}\label{ThDS:1}
Let $G$ be a graph. The following can be determined by its $A_{\alpha}$-spectrum:
\begin{itemize}
\item[\rm (a)] The number of vertices of $G$;
\item[\rm (b)] The number of edges of $G$;
\item[\rm (c)]  Whether $G$ is regular;
\item[\rm (d)]  Whether $G$ is regular with any fixed girth;
\item[\rm (e)] Whether $G$ is bipartite;
\item[\rm (f)] The number of closed walks of any fixed length. In particular, the number of triangles;
\item[\rm (g)] The sum of the square of the degrees of all vertices of $G$;
\item[\rm (h)] The sum of the cube of the degrees of all vertices of $G$.
\end{itemize}
In particular, if $G$ is bipartite, then the following can also be determined by its $A_{\alpha}$-spectrum:
\begin{itemize}
\item[\rm (i)] The number of components of $G$;
\item[\rm (j)] The number of spanning trees of $G$.
\end{itemize}
\end{thm}
\begin{proof}
Items (a)-(f) come from \cite[Lemma 4 (i)-(vi)]{DamH03}. Items (g) and (h) are obtained from (f) and Theorem \ref{Aalphacoeff1} by setting $\alpha=1/2$. Items (i) and (j) come from \cite[Lemma 4 (vii)-(viii)]{DamH03} and the fact \cite[Proposition 2.3]{Cvetkovic07} that a bipartite graph has the same $L$-spectrum and $Q$-spectrum.
\qed\end{proof}

\begin{rem}
{\em Items (a), (b) and (g) of Theorem \ref{ThDS:1} can also be deduced from the $Q$-characteristic polynomial of a graph $G$ (see \cite{Cvetkovic07}, or Corollary \ref{Qcoeff1}).}
\end{rem}

Next, we give some $A_{\alpha}$-DS graphs. The following result can be obtained immediately from the relations between $A(G)$, $Q(G)$ and $A_{\alpha}(G)$.

\begin{prop}
If a graph $G$ is determined by its $A$-spectrum or $Q$-spectrum, then $G$ is determined by its $A_{\alpha}$-spectrum.
\end{prop}

Up until now, many graphs have been proved to be determined by its $A$-spectrum or $Q$-spectrum, for examples, the path \cite[Proposition 1]{DamH03}, the cycle \cite[Proposition 5]{DamH03}, the complete graphs \cite[Proposition 5]{DamH03}, the lollipop graph \cite{BouletJ08, HaemersLZ08, ZhangLZY09}, ect. So, all such graphs are $A_{\alpha}$-DS graphs.

\begin{prop}\label{BproL1}
If $G$ is bipartite and determined by its $L$-spectrum, then $G$ is determined by its $A_{\alpha}$-spectrum.
\end{prop}

\begin{proof}
Let $G'$ be $A_{\alpha}$-cospectral with $G$. Then $G'$ and $G$ are $Q$-cospectral. By (e) of Theorem \ref{ThDS:1}, we obtain that $G'$ is bipartite. Recall that $L$-spectrum and $Q$-spectrum of a bipartite graph are equal. Then $G'$ and $G$ are $L$-cospectral. Since $G$ is determined by its $L$-spectrum, $G'$ is isomorphic to $G$. Thus, $G$ is determined by its $A_{\alpha}$-spectrum. \qed
\end{proof}

A tree is called \emph{starlike} if it has exactly one vertex of degree greater than two.  In \cite{OmidiT07}, all starlike trees were proved to be determined by their $L$-spectra. Thus, Proposition \ref{BproL1} implies the following result immediately.

\begin{cor}
All starlike trees are determined by their $A_{\alpha}$-spectra.
\end{cor}

A tree is called \emph{double starlike} if it has exactly two vertices of degree greater than two. Let $H(p,n,q)$ denote the double starlike tree obtained by attaching $p$ pendant vertices to one pendant vertex of the path $P_n$ and $q$ pendant vertices to the other pendant vertex of $P_n$. In \cite{LiuZL09,LuL12}, it was proved that $H(p,n,q)$ is determined by its $L$-spectrum. Thus, by Proposition \ref{BproL1}, we have the following result readily.

\begin{cor}
$H(p,n,q)$ is determined by its $A_{\alpha}$-spectrum.
\end{cor}


\medskip
\noindent \textbf{Acknowledgements}~~We greatly appreciate Professor V. Nikiforov for his suggestion of studying Paper \cite{Oliveira02} which contributes Examples \ref{Exam1}-\ref{Exam6}.


\end{CJK*}
\end{document}